\documentclass[12pt]{amsart}

\usepackage{amssymb} 
\usepackage{epsfig}
\usepackage{amsmath} 
\usepackage{graphicx}

\newtheorem{thm}{Theorem}[section]

\newtheorem{lm}[thm]{Lemma}

\newtheorem{cor}[thm]{Corollary}
\newtheorem{prop}[thm]{Proposition}
\newtheorem{lem}[thm]{Lemma}

\newcommand{\C}{{{\mathbb C}}}
\newcommand{\R}{{{\mathbb R}}}
\newcommand{\Q}{{{\mathbb Q}}}
\newcommand{\Z}{{{\mathbb Z}}}

\newcommand{\Hh}{{{\mathbb H}}}
\newcommand{\E}{{{\mathcal E}}}
\newcommand{\F}{{{\mathcal F}}}
\newcommand{\Ff}{{{\mathbb F}}}

\newcommand{\Ss}{{{\mathcal S}}}
\newcommand{\A}{{{\mathcal A}}}
\newcommand{\Cay}{{{\rm Cay}}}
\newcommand{\M}{{{\mathcal M}}}
\newcommand{\Aut}{{{\rm Aut}}}
\newcommand{\Sym}{{{\rm Sym}}}
\newcommand{\diam}{{{\rm diam}}}
\newcommand{\ord}{{{\rm ord}}}
\newcommand{\D}{{{\mathbb D}}}

\begin{document}

\title{Trivalent expanders and hyperbolic surfaces}

\author{I. Ivrissimtzis, N. Peyerimhoff and A. Vdovina}

\address{Durham University, DH1 3LE, Great Britain}
\email{ioannis.ivrissimtzis@durham.ac.uk}
\email{norbert.peyerimhoff@durham.ac.uk} 
\address{University of Newcastle, NE1 7RU, Great Britain} 
\email{alina.vdovina@ncl.ac.uk}


\begin{abstract}
  We introduce a family of trivalent expanders which tessellate
  compact hyperbolic surfaces with large isometry groups. We compare
  this family with Platonic graphs and modifications of them and prove
  topological and spectral properties of these families.
\end{abstract}

\maketitle

\section{Introduction and statement of results}

In this article, we consider a family of surface tessellations with
interesting discrete spectral gap properties. More specifically, our
family of graphs, denoted by $T_k$ ($k \ge 2$), are trivalent expander
graphs tessellating hyperbolic surfaces with large isometry groups
growing linear on genus. As we show below, from $k \ge 3$ onwards,
there is no direct relation between our family of graphs and other
graphs associated to modular congruence subgroups.

Let us first give a brief overview over the construction of $T_k$, and
then go into some details. We start with a sequence of $2$-groups
$G_k$, following the construction in \cite[Section 2]{PV}. Then we
consider $6$-valent Cayley graphs $X_k$ of these groups and apply
$(\Delta -Y)$-transformations in all triangles of $X_k$, to finally
obtain the trivalent graphs $T_k$. The $(\Delta-Y)$-transformations
are standard operations to simplify electrical circuits, and were also
used in \cite{BCdV} in connection with Colin de Verdi{\`e}re's graph
parameter. 

The finite groups $G_k$ are constructed as follows. We start with the 
infinite group $\widetilde G$ of seven generators and seven relations:
\begin{equation} \label{eq:G}
\widetilde G = \langle x_0,\dots,x_6 \mid x_i x_{i+1} x_{i+3} \ 
\text{for $i=0,\dots,6$} \rangle, 
\end{equation}
where the indices are taken modulo $7$. As explained in \cite{CMSZ},
this group acts on a thick Euclidean building of type $\tilde
A_2$. Let $S = \{x_0^{\pm 1},x_1^{\pm 1},x_3^{\pm 1}\}$, and consider
the index two subgroup $G \le \widetilde G$, generated by $S$. (Note
that $x_3 = x_1^{-1} x_0^{-1}$.) As explained in \cite[Section 2]{PV},
we use a representation of the group $G$ by infinite (finite band)
upper triangular Toeplitz matrices. The entries of these Toeplitz
matrices are elements of the ring $M(3,{\mathbb F}_2)$ (i.e., $3
\times 3$-matrices over ${\mathbb F}_2$) with special periodicity
properties. We denote the group of all these Toeplitz matrices by $H$,
and by $H_k \le H$ the normal subgroup of matrices whose first $k$
upper diagonals are zero. The groups $G_k$ are then the quotients
$G/(G \cap H_k)$. The {\em finite width conjecture} in \cite{PV}
claims that the groups $G_k$ have another purely abstract group
theoretical description via the {\em lower exponent-$2$ series}
$$ G = P_0(G) \ge P_1(G) \ge P_2(G) \ge \cdots, $$
with $P_k(G) = [P_{k-1}(G),G]P_{k-1}(G)^2$ for $k \ge 1$: namely, $G
\cap H_k = P_k(G)$ for $k \ge 1$ (see \cite[Conj. 1]{PV}). MAGMA
computations confirm this conjecture for all indices up to $k =
100$. For simplicity, we use the same notation for the elements
$x_0,x_1,x_3$ in $G$ and their images in the quotients $G_k$. Then
$X_k = \Cay(G_k,S)$, and $T_k$ are their $(\Delta-Y)$-transformations.

The graphs $T_k$ can be naturally embedded as tessellations into both
compact hyperbolic surfaces $\Ss(T_k)$ and non-compact finite area
hyperbolic surfaces $\Ss_\infty(T_k)$. The edges of the tessellation
are geodesics and the vertices are their end points. Our results are
given in the following theorem:

\begin{thm} \label{thm:main} Let $k \ge 2$. Then every eigenvalue $\mu
  \neq 3$ of $X_k$ gives rise to a pair $\pm \sqrt{\mu+3}$ of
  eigenvalues of the bipartite graph $T_k$. In particular, there
  exists a positive constant $C < 6$ such that
  \begin{itemize}
  \item[(i)] the graphs $X_k$ are $6$-valent expanders with
    spectrum in $[-3,C] \cup \{6\}$,
  \item[(ii)] the bipartite graphs $T_k$ are trivalent expanders with
    spectrum in $[-\sqrt{C+3},\sqrt{C+3}] \cup \{\pm 3 \}$.
  \end{itemize}
  Moreover, the isometry group of the compact hyperbolic surface
  $\Ss(T_k)$ has order $\ge |V(T_k)|/2$, where $V(T_k)$ denotes the
  set of vertices of $T_k$.

  Let
  \begin{equation} \label{eq:rK}
  r = \lfloor \log_2 k \rfloor + 1 \quad \text{and} \ K =
  8 \lfloor k/3 \rfloor + 3\cdot(k\, {\rm mod}\, 3).
  \end{equation} 
  Then we have $|V(T_k)| \ge 2^K$, $|E(T_k)| \ge 3 \cdot 2^{K-1}$ and
  $|F(T_k)| \ge 3 \cdot 2^{K-r-1}$ for the vertices, edges and faces of
  $T_k$, and all faces of $T_k$ are regular $2^{r+1}$-gons. The genus
  of $\Ss(T_k)$ can be estimated by
  $$ g = |V(T_k)| - |E(T_k)| + |F(T_k)| \ge 1 + 2^{K-2} - 3 \cdot 2^{K-r-2}. $$
\end{thm}

The above-mentioned finite width conjecture would imply that the
inequalities for the vertices, edges, faces of $T_k$ and the genus of
$\Ss(T_k)$ in Theorem \ref{thm:main} hold with equality. Theorem
\ref{thm:main} is proved in Section \ref{subsec:main}.

\medskip

It is instructive to compare our tessellations $T_k$ to the well
studied tessellations of hyperbolic surfaces by {\em Platonic graphs}
$\Pi_N$, which are defined as follows. Let $N$ be a positive integer
$\ge 2$. The vertices of $\Pi_N$ are equivalence classes
$[\lambda,\mu] = \{ \pm (\lambda,\mu) \}$ with
$$ \{ (\lambda,\mu) \in \Z_N \times \Z_N \mid \gcd(\lambda,\mu,N)=1 \}. $$
Two vertices $[\lambda,\mu]$ and $[\nu,\omega]$ are connected by an edge
if and only if
$$ \det \begin{pmatrix} \lambda & \nu \\ \mu & \omega \end{pmatrix} = 
\lambda \omega - \mu \nu = \pm 1. $$ 

Note that every vertex of $\Pi_N$ has degree $N$. These graphs can
also be viewed as triangular tessellations of finite area hyperbolic
surfaces $\Ss_\infty(\Pi_N) = \Hh^2 / \Gamma(N)$, where $\Hh^2$
denotes the hyperbolic upper half plane and $\Gamma(N)$ is a principal
congruence subgroup of the modular group $\Gamma = PSL(2,\Z)$. These
and related graphs have been thoroughly investigated by several
different communities. For example, in the general framework of
{\em regular maps}, they were studied by D. Singerman and co-authors (see
\cite{JSi,Si,ISi}). For odd prime numbers $N=p$, the graphs $\Pi_p$
have maximal vertex connectivity $p$, diameter $3$, and are Ramanujan
graphs. Analogous properties hold for the induced subgraphs $\Pi_p'$,
where $\Pi_p'$ is obtained from $\Pi_p$ by removing the set of
vertices $[\lambda,0]$ with vanishing second coordinate and all their
adjacent edges. Note that $\Pi_p'$ is a $(p-1)$-valent graph
tessellating the same surface $\Ss_\infty(\Pi_p)$. As in the case of
our family $T_k$, the graphs $\Pi_N$ and $\Pi_N'$ can also be embedded
into smooth compact hyperbolic surfaces, denoted by $\Ss(\Pi_N)$.

It turns out that the graphs $T_2^*$ and $\Pi_8$ are isomorphic. Since
the valence of the dual graph $T_k^*$ is a power of $2$, any
isomorphism of $T_k^*$ with a Platonic graph $\Pi_N$ would imply $N =
2^\rho$ with $\rho = \lfloor \log_2 k \rfloor + 2$. However, this
leads to a contradiction {\em for all} $k \ge 3$. The next proposition
summarizes the comparison between our graphs and Platonic graphs
showing that, generally, {\em these two families are of very different
  nature}. The proof is given in Section \ref{subsec:isomT2Pi8}.

\begin{prop} \label{prop:isomt2s8}
  The graph $T_2$ is the dual of the Platonic graph $\Pi_8$ in the
  {\em unique} genus 5 hyperbolic surface $\Ss(T_2)=\Ss(\Pi_8)$ with
  maximal automorphism group of order 192. For $k \ge 3$, there is no
  graph isomorphism between $T_k^*$ and $\Pi_N$, for any
  $N$.
\end{prop}

Let us say a few more words about the Platonic graphs $\Pi_p$ and
their modifications $\Pi_p'$. The modified graphs $\Pi_p'$ have an
alternative description as Cayley graphs of the quotients
$\Gamma_0(p)/\Gamma(p)$ of congruence subgroups (see end of Section
\ref{subsec:vertconnect}). However, we did not find these modified
graphs explicitly in the literature (for example, they do not appear
explicitly in \cite[Cor. 8.2.3]{Lub} or in \cite{Te1,WLM}). {\em
  Cheeger constant estimates} for $\Pi_p$ have been obtained, e.g., in
\cite{BPP,LR}, but we do not know of any reference for the {\em
  maximal vertex connectivity}, and present a proof of this fact for
both graph families $\Pi_p$ and $\Pi_p'$ in Section
\ref{subsec:vertconnect}. To our knowledge, all proofs for the {\em
  Ramanujan property} of the graphs $\Pi_p$ in the literature (see,
e.g., \cite{Gunnells,WLM,DDLM}) are based on some amount of number
theory (characters of representations). We think it is remarkable that
there is also an easy proof for the Ramanujan properties of the graphs
$\Pi_p$ and $\Pi_p'$ with no reference to number theory other than the
irrationality of $\sqrt{p}$ (see Section
\ref{subsec:ramanujan}). These facts are summarized in the following
theorem.

\begin{thm} \label{thm:platonic} Let $p$ be an odd prime. Then the
  graphs $\Pi_p$ and $\Pi_p'$ have diameter $3$ and maximal vertex
  connectivity $p$ and $p-1$, respectively.  Moreover, the spectrum of
  the graph $\Pi_p'$ consists of
  \begin{itemize}
  \item[(i)] $p-1$ with multiplicity one,
  \item[(ii)] $-1$ with multiplicity $p-1$,
  \item[(iii)] $0$ with multiplicity $(p-3)/2$, and
  \item[(iv)] $\pm \sqrt{p}$ with multiplicity $(p-1)(p-3)/4$, each.
  \end{itemize}
  In particular, the graphs $\Pi_p'$ are Ramanujan.
\end{thm}

As mentioned earlier, our family $T_k$ is based on powers of the prime
number 2. The question whether the Ramanujan property of $\Pi_N$ for
primes $N=p$ still holds for composite numbers $N$ or, at least, for
prime powers $N = p^r$, was answered in the negative in
\cite[Prop. 4.7]{Gunnells}. But the Ramanujan property for prime
powers is preserved if one considers Platonic graphs over finite
fields $\Ff_{p^r}$ instead of the rings $\Z_{p^r}$ (see \cite{DDLM}).
However, we do not know how these Ramanujan graphs associated to prime
powers could be naturally embedded into appropriate surfaces.

It is easily checked that any triangular tessellation $X$ of a compact
oriented surface $\Ss$ satisfies $|E(X)| = 3(|V(X)|-2) + 6g(\Ss)$,
i.e., the number of edges of every triangulation with at least two
vertices is $\ge 6g(\Ss)$.  Therefore, the ratio
$$ \frac{6g(\Ss)}{|E(X)|} \le 1 $$ 
measures the non-flatness of such a triangulation, i.e., how
effectively the edges of $X$ are chosen to generate a surface of high
genus. The following asymptotic results hold for the triangulations
$\Pi_N$ and $T_k^*$.

\begin{prop} \label{prop:asymptnonflat}
  We have
  \begin{equation} \label{eq:asympPiN} 
  \lim_{N \to \infty} \frac{6g(\Ss(\Pi_N))}{|E(\Pi_N)|} = 1, 
  \end{equation}
  and
  \begin{equation} \label{eq:asympTk}
  \lim_{k \to \infty} \frac{6g(\Ss(T_k))}{|E(T_k^*)|} = 1. 
  \end{equation}
  In the second case, note that the dual graph $T_k^*$ is a
  triangulation of $\Ss(T_k)$ and that the number of edges of $T_k$
  and $T_k^*$ coincide.
\end{prop}

The two formulas in this proposition are proved in Sections
\ref{subsec:prop12a} and \ref{subsec:prop12b}.

Our trivalent expander graphs $T_k$ can also be used to construct
another family of compact hyperbolic surfaces $\widehat \Ss(T_k)$ by
glueing together regular $Y$-pieces, as explained in Buser
\cite{Bu}. The surfaces $\widehat \Ss(T_k)$ can be viewed as tubes
around the graphs $T_k$ with a hyperbolic metric. Using the results in
\cite{Bu}, the expander properties of $T_k$ translate directly into a
uniform lower bound of the first non-trivial eigenvalue $\lambda_1$ of
the Laplacian on these surfaces.

\begin{cor} \label{cor:lowerlambda1}
  The compact hyperbolic surfaces $\widehat \Ss(T_k)$ ($k \ge 2$) have
  genus $1+|V(T_k)|/2$ and isometry groups of order $\ge
  |V(T_k)|/2$. They form a tower of coverings
  $$ \cdots \longrightarrow \widehat \Ss(T_{k+1}) \longrightarrow \widehat 
  \Ss(T_k) \longrightarrow \widehat \Ss(T_{k-1}) \longrightarrow
  \cdots $$ 
  where all the covering indices are powers of $2$. There is a positive
  constant $\epsilon > 0$ such that we have, for all $k$,
  $$ \lambda_1(\widehat \Ss(T_k)) \ge \epsilon. $$
\end{cor}

Corollary \ref{cor:lowerlambda1} is proved in Section
\ref{subsec:corlambda1}. There is a well-known classical result by
Randol \cite{Ran} which is in some sense complementary to this
corollary. Namely, there exist finite coverings $\widetilde \Ss$ of
every compact hyperbolic surface $\Ss$ with arbitrarily small first
eigenvalues. It would be interesting to find out whether there are
also uniform positive lower bounds for $\lambda_1$ of our other
compact hyperbolic surfaces $\Ss(T_k)$. It seems that the methods in
\cite{Br1,Br2} are not applicable in this case, since the shapes of
the hyperbolic triangles tessellating these surfaces are changing with
$k$.

\medskip

\noindent {\bf Acknowledgement:} We like to thank Hugo Parlier for
useful discussions.

\section{Properties of the tessellations $(T_k,\Ss(T_k))$}

\subsection{The surfaces $\Ss_\infty(T_k)$ and $\Ss(T_k)$}
\label{subsec:surfs}

The explicit construction of the Cayley graphs $X_k = \Cay(G_k,S)$ and
of the trivalent graphs $T_k$ was explained in the introduction. For
the underlying $2$-groups $G_k$, we refer the reader to \cite[Section
2]{PV}. Let us now construct the hyperbolic surface $\Ss_\infty(T_k)$:
We start with a 3-punctured sphere $\Ss_0$, by glueing together two
ideal hyperbolic triangles along their corresponding edges. Note that
$\Ss_0$ carries a hyperbolic metric. It is useful to think of the two
ideal triangles of $\Ss_0$ to be coloured black and white. Let $P_0
\in \Ss_0$ be the center of the black triangle. Choose a geometric
basis $\gamma_0,\gamma_1,\gamma_2 \in \pi(\Ss_0,P_0)$ such that
$\gamma_i$ is a simple counterclockwise look around the $i$-th cusp of
$\Ss_0$ and $\gamma_0 \gamma_1 \gamma_2 = e$. The surjective
homomorphism
$$ \Psi: \pi(\Ss_0,P_0) \to G_k, $$
given by $\Psi(\gamma_0) = x_0$, $\Psi(\gamma_1) = x_1$ and
$\Psi(\gamma_2) = x_3$, induces a Riemannian covering map $\pi:
\Ss_\infty \to \Ss_0$. The surface $\Ss_\infty$ is a hyperbolic
surface, tessellated by $2|G_k|$ ideal hyperbolic triangles, half
of them black and the others white. Hurwitz's formula yields
\begin{equation} \label{eq:gSinf}
g(\Ss_\infty) = 1 + \frac{1-\mu_k}{2} |G_k|, 
\end{equation}
where
\begin{equation} \label{eq:mu}
  \mu_k = \frac{1}{\ord(x_0)} + \frac{1}{\ord(x_1)} + \frac{1}{\ord(x_3)}.
\end{equation}
In the case $k=2$ we have $|G_2| = 32$ and
$\ord(x_0)=\ord(x_1)=\ord(x_3)=4$, which leads to 
$$ g(\Ss_\infty) = 1 + \frac{1}{8} \cdot 32 = 5. $$

$G_k$ acts simply transitive on the black triangles of
$\Ss_\infty$. Let $V = \pi^{-1}(P_0)$ and $V_{\rm black}, V_{\rm
  white} \subset V$ be the sets of centers of black and white
triangles, respectively. Choose a reference point $P \in V_{\rm
  black}$, and identify the vertices of the Cayley graph $X_k$
with the points in $V_{\rm black}$ by $G_k \ni h \mapsto hP \in
V_{\rm black}$. Then two adjacent vertices in $X_k$ are the
centers of two black triangles which share a white triangle as their
common neighbour. The corresponding edge is then the minimal geodesic
passing through these three ideal triangles and connecting these two
vertices.

We could instead start the process by glueing together two {\em
  compact} hyperbolic triangles with angles $\pi/\ord(x_0)$,
$\pi/\ord(x_1)$ and $\pi/\ord(x_3)$, and obtain an orbifold
$\Ss_0$. The same arguments then lead to an embedding of
$X_k$ into a smooth compact hyperbolic surface $\Ss$ with the same
genus as $\Ss_\infty$. The surface $\Ss$ is triangulated by compact black and
white triangles, and every black triangle contains a vertex of
$X_k$. In fact, $G_k$ acts on the surface $\Ss$ by isometries.

The abstract $(\Delta-Y)$-transformation of a graph adds a new vertex
$v$ for every triangle, removes the three edges of this triangle and
replaces them by three edges connecting $v$ with the vertices of this
triangle. We apply this rule to our graph $X_k$ and obtain a graph
$T_k$, which we can view as an embedding in $\Ss$ with the following
properties: The vertex set of $T_k$ coincides with $V$, and there is
an edge (geodesic segment) connecting every black/white vertex in $V$
with the vertices in the three neighbouring white/black triangles. The
best way to illustrate this transformation is to present it in the

universal covering of the surface $\Ss$, i.e., the Poincar{\'e} unit
disc $\D$ (see Figure \ref{fig:DYtrafo}, the new vertices replacing
every triangle are green). Note that $T_k$ has twice as many vertices
as $X_k$, which shows that the isometry group of the above compact
surface $\Ss = \Ss(T_k)$ has order $\ge |G_k| = |V(T_k)|/2$.
Moreover, $T_k$ is indeed the dual of the triangulation of $\Ss$ by
the abovementioned compact black and white triangles.

\begin{figure}[h]
  \begin{center}      
     \includegraphics[width=0.4\textwidth]{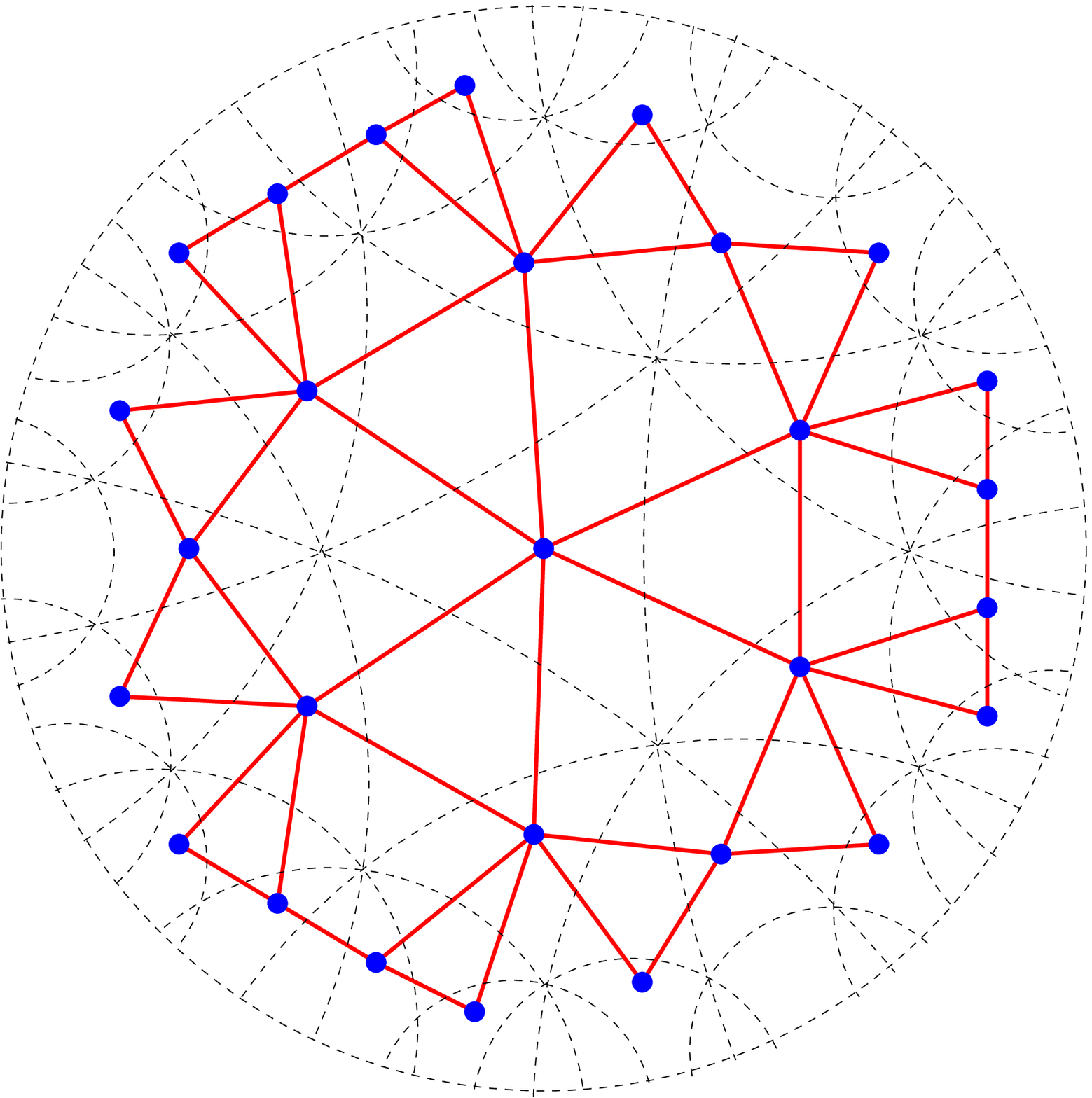}
     \includegraphics[width=0.4\textwidth]{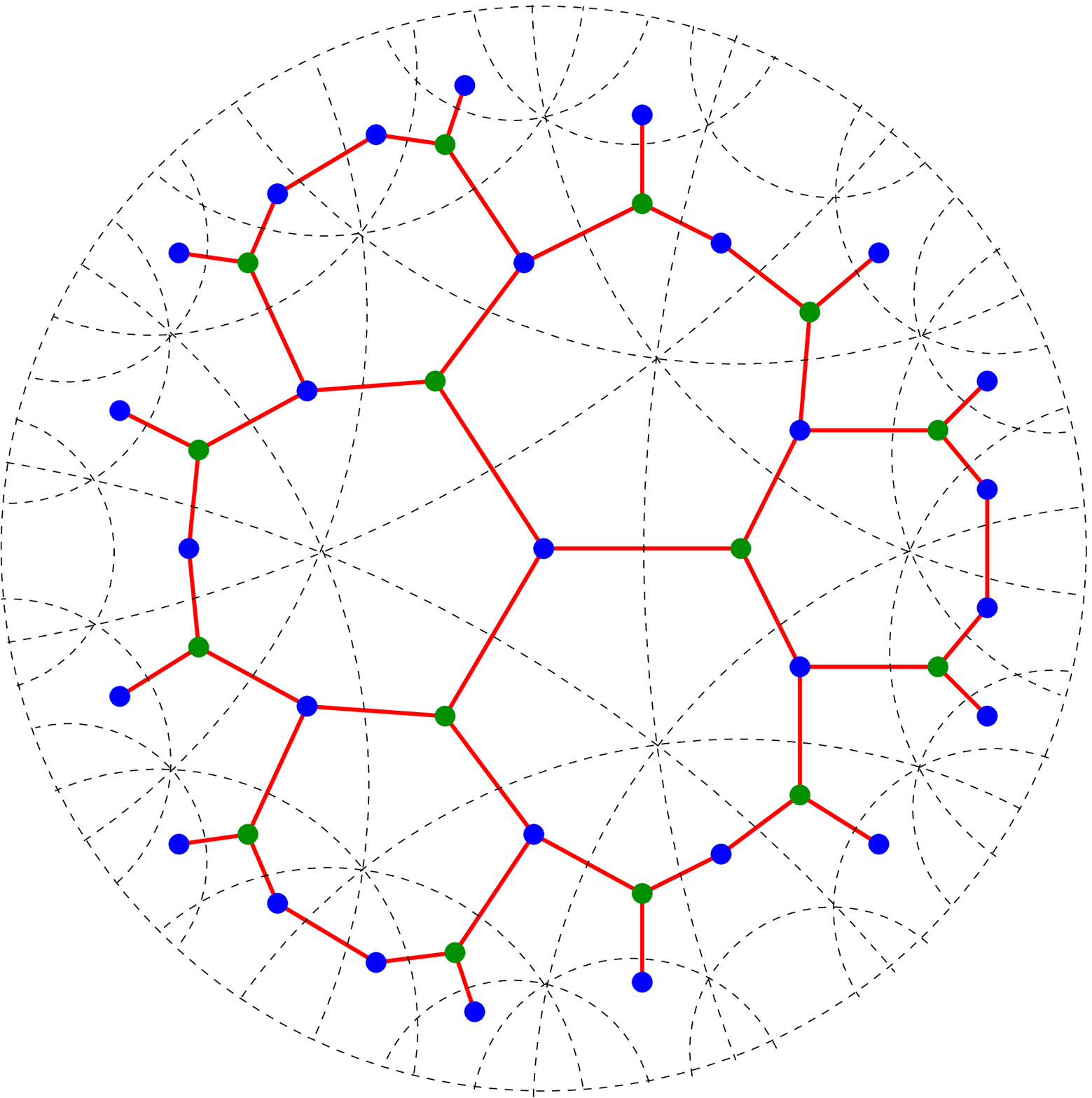}
  \end{center}
  \caption{The lifts of the Cayley graph $X_2$ (left) and of the
    $(\Delta-Y)$-transformation $T_2$ (right) to the Poincar{\'e} unit
    disc $\D$}
   \label{fig:DYtrafo}
\end{figure}

\subsection{Proof of Theorem \ref{thm:main}}
\label{subsec:main}

We first establish the expander properties of $X_k$ and $T_k$ and the
relations between their eigenvalues, stated in Theorem
\ref{thm:main}. It was proved in \cite[Section 2]{PV} that the group
$G$ generated by $x_0,x_1 \in \widetilde G$ is an index two subgroup
of the group $\widetilde G$ in \eqref{eq:G}. $G$ is explicitly given
by $G = \langle x_0, x_1 \mid r_1,r_2,r_3 \rangle$ with
\begin{eqnarray}
  r_1(x_0,x_1) &=& (x_1 x_0)^3x_1^{-3}x_0^{-3}, \nonumber \\
  r_2(x_0,x_1) &=& x_1x_0^{-1}x_1^{-1}x_0^{-3}x_1^2x_0^{-1}x_1x_0x_1, 
  \label{eq:r1r2r3} \\
  r_3(x_0,x_1) &=& x_1^3x_0^{-1}x_1x_0x_1x_0^2x_1^2x_0x_1x_0. \nonumber
\end{eqnarray}
(Note that our group $\widetilde G$ is denoted in \cite{PV} by
$\Gamma$, which is reserved for $PSL(2,\Z)$ in this paper.)  Moreover,
both groups $\widetilde G$ and $G$ have Kazhdan property (T) (see
\cite[Section 3]{PV}). Using \cite[Prop. 3.3]{Lub}, we conclude that
the Cayley graphs $X_k$ are expanders.

The adjacency operator $A$, acting on functions on the vertices of a
graph, is defined as
$$ Af(v) = \sum_{w \sim v} f(w). $$
Note that $V(X_k)$ is a subset of $V(T_k)$. We have the following
relations between the eigenfunctions of the adjacency operators on
$X_k$ and $T_k$.

\begin{thm} \label{thm:releigfunc}
  \begin{itemize}
  \item[(a)] Every eigenfunction $F$ on $T_k$ to an eigenvalue
    $\lambda \in [-3,3]$ gives rise to an eigenfunction $f$ to the
    eigenvalue $\mu = \lambda^2 - 3 \in [-3,6]$ on $X_k$ (with
    $f(v) = F(v)$ for all $v \in V(X_k)$).
  \item[(b)] Every eigenfunction $f$ on $X_k$ to an eigenvalue
    $\mu \in [-6,6] - \{3\}$ gives rise to two eigenfunctions $F_\pm$
    to the eigenvalues $\pm \sqrt{\mu + 3}$ on $T_k$ with
    $$ F_\pm(v) = \begin{cases} f(v) & \text{if $v \in V(X_k)$,} \\
      \pm \frac{1}{\sqrt{\mu+3}} \sum_{w \sim v} f(w) & \text{if $v \in
          V(T_k)-V(X_k)$.} \end{cases}
    $$
  \item[(c)] An eigenfunction $f$ on $X_k$ to the eigenvalue $-3$
    gives rise to an eigenfunction $F$ to the eigenvalue $0$ of $T_k$ with
    $$ F(v) = \begin{cases} f(v) & \text{if $v \in V(X_k)$,} \\
      0 & \text{if $v \in V(T_k)-V(X_k)$,} \end{cases}
    $$
    if and only if we have, for all triangles $\Delta$ in $X_k$,
    $\sum_{v \in V(\Delta)} f(v) = 0$.
  \end{itemize}
\end{thm}

\begin{proof}
  (a) Let $f$ and $F$ be two functions on $X_k$ and $T_k$, related
  by $f(v) = F(v)$ for all $v \in V(X_k)$. Then
  $$ A_{X_k}f(v) = \sum_{w \sim_{X_k} v} f(w) = 
  \sum_{d_{T_k}(w,v) = 2} F(w) = (A_{T_k})^2 F(v) - 3 F(v), $$ which
  can also be written as $A_{X_k} = (A_{T_k})^2 - 3$. (Note that
  $\sim_{X_k}$ denotes adjacency in $X_k$, and $d_{T_k}$ is the
  combinatorial distance in $T_k$.) This implies immediately the
  connection between the eigenfunctions and eigenvalues.

  (b) Let $A_{X_k}f = \mu f$ and $F_\pm$ be defined as in the theorem.
  Let $\lambda = \pm \sqrt{\mu+3}$. Then we have for $v \in V(X_k)$:
  \begin{eqnarray*}
    A_{T_k}F_\pm(v) &=& \sum_{w \sim v} F_\pm(w) = \frac{1}{\lambda} \sum_{w \sim v}
    \sum_{x \sim w} F_\pm(x)\\
    &=& \frac{1}{\lambda} \left( 
      \sum_{w \sim_{X_k} v} f(w) + 3 f(v) \right)
    = \frac{\mu+3}{\lambda} f(v) = \lambda F_\pm(v), 
  \end{eqnarray*}
  and for $v \in V(T_k) - V(X_k)$:
  $$ A_{T_k}F_\pm(v) = \sum_{w \sim v} F_\pm(w) = 
  \lambda \left( \frac{1}{\lambda} \sum_{w \sim v} f(w) \right) =
  \lambda F_\pm(v). $$ 
  Note that $1/\lambda$ is well defined since $\lambda = \pm \sqrt{3}
  \neq 0$. 

  (c) In the case of $\mu = -3$ we have $\lambda=0$, and the above
  calculation for $v \in V(X_k)$ goes through without changes. For
  $v \in V(T_k) - V(X_k)$, the condition
  $$ 0 = A_{T_k}F_\pm(v) = \sum_{w \sim_{T_k} v} f(v) $$
  translates into the condition that the summation of $f$ over the
  vertices of all the triangles must vanish.
\end{proof}

Theorem \ref{thm:releigfunc} implies that the expander property of the
family $X_k$ carries over to the graphs $T_k$. Moreover, the
spectrum of $X_k$ cannot contain eigenvalues in the interval
$[-6,-3)$, since this would lead to non-real eigenvalues of
$T_k$. This finishes the proof of the spectral statements in Theorem
\ref{thm:main}.

Since $X_k$ are Cayley graphs of quotients of the group $G$ with
property (T), not all of these graphs can be Ramanujan (see
\cite[Prop. 4.5.7]{Lub}). But what can we say about their
$(\Delta-Y)$-transformations $T_k$? Theorem \ref{thm:releigfunc}
implies that $T_k$ is Ramanujan if and only if the largest non-trivial
eigenvalue of $X_k$ is $< 5$. MAGMA computations provide the
following numerical results:

\medskip

\begin{table}[htbp]
\begin{tabular}{l|l|l}
graph & number of vertices & largest non-trivial eigenvalue \\
\hline\rule[1.8mm]{0cm}{2mm}
$X_2$ & 32 &   2.828427124746190\dots \\ \rule[1.5mm]{0cm}{2mm}
$X_3$ & 128 &  4.340172973252067\dots \\ \rule[1.5mm]{0cm}{2mm}
$X_4$ & 1024 & 4.475244292138809\dots \\ \rule[1.5mm]{0cm}{2mm}
$X_5$ & 8192 & 5.160252515773351\dots
\end{tabular}
\end{table}

This implies that only $X_2$ and $X_3$ are Ramanujan; their
largest non-trivial eigenvalue needs to be $< 2 \sqrt{5} =
4.472135\dots$, which is no longer true for $k = 4$. Moreover, we have
$\sigma(X_k) \subset \sigma(X_{k+1})$, since the graphs
$X_k$ are a tower of coverings. Similarly, only $T_2, T_3, T_4$
are Ramanujan, since the covering properties of $X_k$ carry over
to their $(\Delta-Y)$-transformations $T_k$.

Finally, we obtain from \cite[Cor. 2.3]{PV} that 
\begin{equation} \label{eq:VTk}
|V(T_k)| = 2 |G_k| = 2 [G:(G \cap H_k)] \ge 2^K. 
\end{equation}
Since Conjecture 1 in \cite{PV} (i.e., $G \cap H_k = P_k(G)$) holds
for all $3 \le k \le 100$, \eqref{eq:VTk} holds for these indices with
equality. The faces of $T_k$ are determined by the orders of the
generators $x_0,x_1,x_3$ in the group $G_k$, which we determine
next. Let {\small
\begin{align*}
& \alpha_0 = \begin{pmatrix} 0 & 0 & 0 & 0 & 0 & 0 & 0 & 0 & 0\\ 
                   0 & 0 & 1 & 0 & 0 & 1 & 0 & 0 & 1\\ 
                   0 & 1 & 1 & 0 & 1 & 1 & 0 & 1 & 1 \end{pmatrix},
& \beta_0 = \begin{pmatrix} 0 & 0 & 0 & 0 & 0 & 0 & 0 & 0 & 0\\ 
                   0 & 1 & 1 & 0 & 1 & 1 & 0 & 1 & 1\\ 
                   0 & 1 & 0 & 0 & 1 & 0 & 0 & 1 & 0 \end{pmatrix},\\[.2cm]
& \alpha_1 = \begin{pmatrix} 0 & 0 & 0 & 0 & 1 & 1 & 0 & 1 & 0\\ 
                   0 & 1 & 0 & 1 & 0 & 0 & 0 & 0 & 1\\ 
                   1 & 1 & 1 & 0 & 0 & 0 & 0 & 1 & 0 \end{pmatrix},
& \beta_1 = \begin{pmatrix} 0 & 0 & 0 & 0 & 1 & 1 & 0 & 1 & 0\\ 
                   0 & 1 & 0 & 1 & 0 & 0 & 0 & 0 & 1\\ 
                   1 & 1 & 1 & 0 & 0 & 0 & 0 & 1 & 0 \end{pmatrix},\\[.2cm]
& \alpha_3 = \begin{pmatrix} 0 & 0 & 0 & 0 & 1 & 1 & 0 & 1 & 0\\ 
                   0 & 1 & 1 & 1 & 0 & 1 & 0 & 0 & 0\\ 
                   1 & 0 & 0 & 0 & 1 & 1 & 0 & 0 & 1 \end{pmatrix},
& \beta_3 = \begin{pmatrix} 0 & 0 & 0 & 0 & 0 & 1 & 0 & 1 & 1\\ 
                   1 & 1 & 0 & 0 & 1 & 1 & 0 & 0 & 0\\ 
                   0 & 1 & 1 & 0 & 0 & 1 & 1 & 0 & 0 \end{pmatrix}.
\end{align*}
}

\begin{lem} \label{lm:ordersxi}
  We have, in the notation of \cite{PV}, for $i \in \{0,1,3\}$:
  \begin{equation} \label{eq:pow2} 
  x_i^{2^l} = \begin{cases} M_{2^l-1}(\alpha_i,\dots) & \text{if $l$ is even}, \\
    M_{2^l-1}(\beta_i,\dots) & \text{if $l$ is odd}. \end{cases}
  \end{equation}
  This implies, in particular, that $\ord_{G_k}(x_i) = 2^r$ with $r$ given in \eqref{eq:rK}.
\end{lem}

\begin{proof}
  Since $G_k$ is a $2$-group, $\ord_{G_k}(x_i)$ has to be a power of
  $2$. The formulas \eqref{eq:pow2} follow from a straightforward
  calculation from $x_i = M_0(\alpha_i,\dots)$ and Proposition 2.5 in
  \cite{PV}. This implies that $\ord_{G_k}(x_i) = 2^r$ if and only if
  $2^{r-1}-1 \le k-1 < 2^r-1$, i.e., $r = \lfloor \log_2 k \rfloor +
  1$.
\end{proof}

Lemma \ref{lm:ordersxi} implies that the faces of $T_k$ are regular
$2^{r+1}$-gons. The estimates for $|E(T_k)|$ and $|F(T_k)|$ in Theorem
\ref{thm:main} follow immediately from this fact, the inequality
\eqref{eq:VTk}, and the trivalence of the graphs $T_k$. The genus
estimate for $\Ss(T_k)$ can then be deduced from \eqref{eq:gSinf} and
\eqref{eq:mu}. This finishes the proof of Theorem \ref{thm:main}.

\subsection{Proof of \eqref{eq:asympTk} in Proposition
  \ref{prop:asymptnonflat}}
\label{subsec:prop12a}

We conclude from the trivalence of $T_k$ and \eqref{eq:gSinf} that
$$ \frac{6g(\Ss(T_k))}{|E(T_k^*)|} = 6 \,
\frac{1+(1-\mu_k)|V(T_k)|/4}{3 |V(T_k)|/2}. $$
Note that $|V(T_k)| = 2|H_k| \ge 2^K \to \infty$, which implies that
$$ \lim_{k \to \infty} \frac{6g(\Ss(T_k))}{|E(T_k^*)|} = 
1 - \lim_{k\to\infty} \mu_k. $$ 
Recall from \eqref{eq:mu} and Lemma \ref{lm:ordersxi} that $\mu_k =
3/\ord_{G_k}(x_0) \to 0$ as $k \to \infty$, finishing the proof of
\eqref{eq:asympTk}.

\subsection{Proof of Proposition \ref{prop:isomt2s8}}
\label{subsec:isomT2Pi8}

We first recall a few important facts about the Platonic graphs
$\Pi_N$ and the surfaces $\Ss_\infty(\Pi_N)$ and $\Ss(\Pi_N)$. For
more details, see \cite{ISi}. Let $\F$ be the Farey tessellation of
the hyperbolic upper half plane $\Hh^2$, and let $\Omega(\F)$ be the
set of oriented geodesics in $\F$. Recall that the {\em Farey
  tessellation} is a triangulation of $\Hh^2$ with vertices on the
line at infinity $\R \cup \{ \infty \}$, namely, the subset of
extended rationals $\Q \cup \{ \infty \}$. Two rational vertices with
reduced forms $a/c$ and $b/d$ are joined by an edge, a geodesic of
$\Hh^2$, if and only if $ad-bc = \pm 1$ (see \cite[Fig. 1]{ISi} for an
illustration of the Farey tessellation). The group of conformal
transformations of $\Hh^2$ that leave $\F$ invariant is the modular
group $\Gamma = PSL(2,\Z)$, which acts transitively on
$\Omega(\F)$. The {\em principal congruence subgroups} of $\Gamma$ are
normal subgroups defined by
$$ \Gamma(N) = \left\{ \begin{pmatrix} a & b \\ c & d \end{pmatrix} \in \Gamma 
  \mid \begin{pmatrix} a & b \\ c & d \end{pmatrix} \equiv
  \pm \begin{pmatrix} 1 & 0 \\ 0 & 1 \end{pmatrix} \ {\rm mod}\ N
\right\}. $$ 

It is well known (see, e.g, \cite{ISi}) that $\F/\Gamma(N)$ and
$\Pi_N$ are isomorphic, and $\F/\Gamma(N)$ is a triangulation of the
surface $\Ss_\infty(\Pi_N) = \Hh^2/\Gamma(N)$ by ideal triangles (the
vertices are, in fact, the cusps of $\Ss_\infty(\Pi_N)$). The
tessellation $(\Pi_N,\Ss_\infty(\Pi_N))$ can be interpreted as a {\em
  map} $\M_N$ in the sense of Jones/Singerman \cite{JSi}. The group
$\Aut(\M_N)$ of automorphisms of $\M_N$ is the group of orientation
preserving isometries of $\Ss_\infty(\Pi_N)$ preserving the
triangulation. As $\Gamma(N)$ is normal in $\Gamma$, we have that the
map $\M_N$ is {\em regular}, meaning that $\Aut(\M_N)$ acts
transitively on the set of directed edges of $\Pi_N$ (see \cite[Thm
6.3]{JSi}). Moreover, by \cite[Thm 3.8]{JSi},
$$ \Aut(\M_N) \cong \Gamma/\Gamma(N) \cong PSL(2,\Z_N). $$
(Note that in the case of a prime power $N = p^r$, $PSL(2,\Z_N)$ is
the group defined over the ring $\Z_N$ and not over the field with
$p^r$ elements.) Let $N \ge 7$. Noticing that all vertices of $\Pi_N$
have degree $N$, we obtain a smooth compact surface $\Ss(\Pi_N)$
by substituting every ideal triangle in $(\Pi_N,\Ss_\infty(\Pi_N))$ by
a compact hyperbolic $(2\pi/N,2\pi/N,2\pi/N)$-triangle, and glueing them
along their edges in the same way as the ideal triangles of
$\Ss_\infty(\Pi_N)$. The group of orientation preserving isometries of
$\Ss(\Pi_N)$ preserving this triangulation is, again, isomorphic to
$PSL(2,\Z_N)$. Hence, the automorphism group of the triangulation
$(\Pi_8,\Ss(\Pi_8))$ is $PSL(2,\Z_8)$ of order 192. This implies that
$\Ss(\Pi_8)$ is the unique compact hyperbolic surface of genus 5 with
maximal automorphism group (see \cite{Battsengel}). 

\begin{figure}[htbp]
\begin{center}
\includegraphics[width=\textwidth]{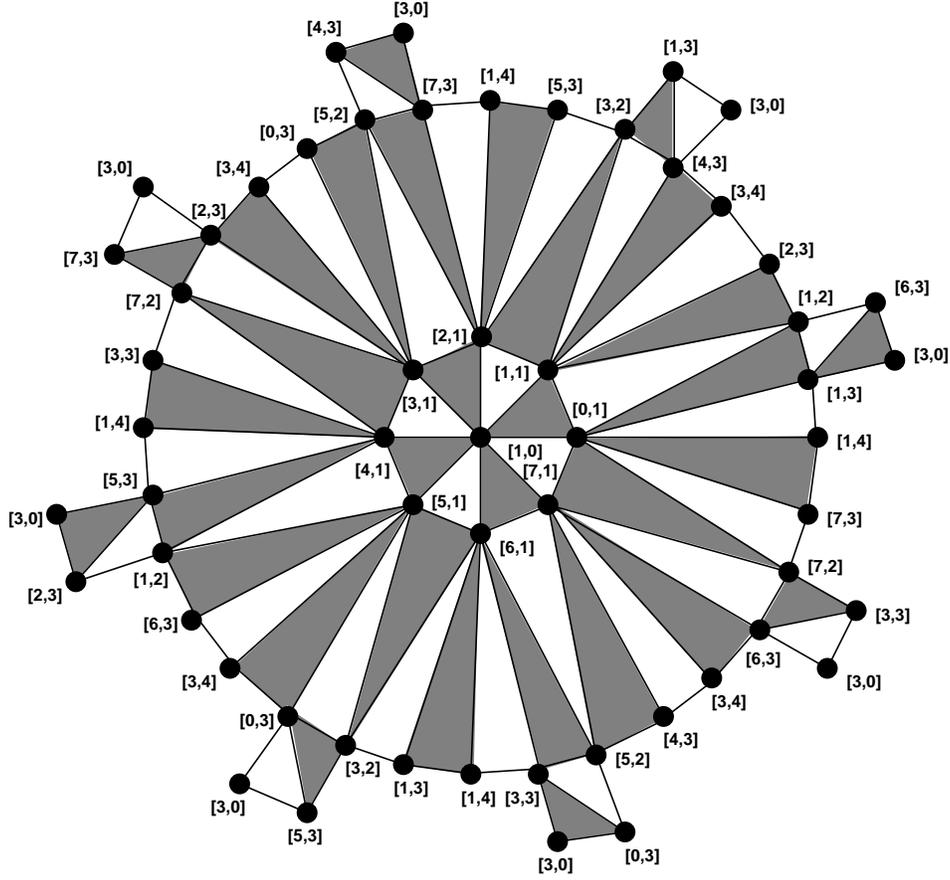}
\caption{The Platonic graph $\Pi_8$: Each triangle corresponds to a
  hyperbolic $(\pi/4,\pi/4,\pi/4)$-triangle of the tessellation of
  $\Ss(S_8)$. The edges along the boundary path are pairwise glued to
  obtain $\Ss(\Pi_8)$.}
\label{fig:s8}
\end{center}
\end{figure}

The $\Pi_8$-triangulation of $\Ss(\Pi_8)$ is shown in Figure
\ref{fig:s8}; the black-white pattern on the triangles is a first test
whether this triangulation can be isomorphic to the
$T_2^*$-triangulation of $\Ss(T_2)$. (The $\Pi_N$-triangulations for
$3 \le N \le 7$ can be found in Figs. 3 and 4 of \cite{ISi}.)
$PSL(2,\Z_8)$ acts simply transitively on the directed edges of this
triangulation. Consider now a refinement of this triangulation by
subdividing each $(\pi/4,\pi/4,\pi/4)$-triangle into six
$(\pi/2,\pi/3,\pi/8)$-triangles. It is easily checked that the smaller
$(\pi/2,\pi/3,\pi/8)$-triangles admit also a black-white colouring
such that the neighbours of all smaller black triangles are white
triangles and vice versa. Each black $(\pi/2,\pi/3,\pi/8)$-triangle is
in 1-1 correspondence to a half-edge of $\Pi_8$ which, in turn, can
be identified with a directed edge of $\Pi_8$. Consequently, the
orientation preserving isometries of the surface $\Ss(\Pi_8)$
corresponding to the elements in $PSL(2,\Z_8)$ act simply transitively
on the black $(\pi/2,\pi/3,\pi/8)$-triangles. In fact, $PSL(2,\Z_8)$
can be interpreted as a quotient of the triangle group
$\Delta^+(2,3,8)$, namely,
$$ PSL(2,\Z_8) \cong \langle x^2,y^3,z^8,xyz,(xz^2xz^5)^2 \rangle, $$ 
where $x,y,z$ correspond to rotations by $\pi,2\pi/3,\pi/4$ about the
three vertices of a given $(\pi/2,\pi/3,\pi/8)$-triangle. 

MAGMA computations show that $PSL(2,\Z_8)$ has a unique normal
subgroup $N$ of index 6, generated by the elements $X = x^{-1}z^2x$,
$Y = y^{-1}z^2y$ and $Z = z^2$, which is isomorphic to the triangle
group quotient $\Delta^+(4,4,4)/P_2(\Delta^+(4,4,4))$ via the explicit
isomorphism
\begin{equation} \label{eq:XYZ} X \mapsto \begin{pmatrix} -1 & 0 \\ 2
    & -1 \end{pmatrix}, \quad Y \mapsto \begin{pmatrix} -1 & 2 \\ -2 &
    3 \end{pmatrix}, \quad Z \mapsto \begin{pmatrix} 1 & 2 \\ 0 &
    1 \end{pmatrix}.
\end{equation}
Note that the matrices in \eqref{eq:XYZ}, viewed as elements in
$PSL(2,\Z)$, generate a group acting simply transitively on the black
triangles of the Farey tessellation in $\Hh^2$, as illustrated in
Figure \ref{fig:fargen}. The images of a black triangle ${\mathcal T}$
with vertices $0,1,\infty$ under $\{ X^{\pm 1}, Y^{\pm 1}, Z^{\pm 1}
\}$ are the six black triangles each sharing a common white triangle
with ${\mathcal T}$.

\begin{figure}[htbp]
\begin{center}
\includegraphics[width=0.8\textwidth]{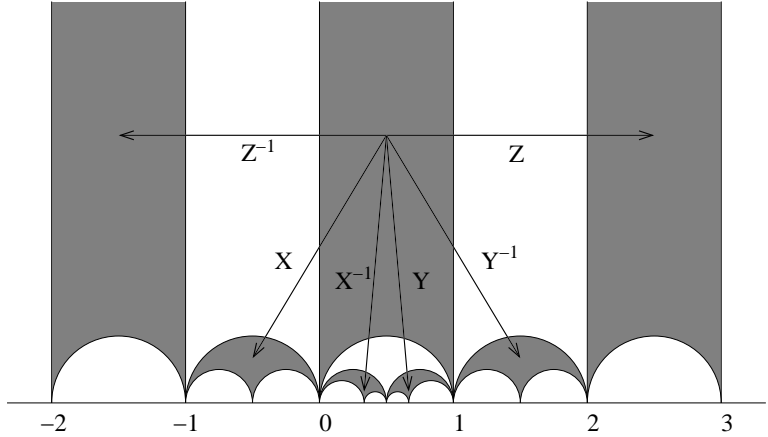}
\caption{The action of the elements $X^{\pm 1},Y^{\pm 1},Z^{\pm 1} \in
  PSL(2,\Z)$ on a triangle ${\mathcal T}$ with vertices $0,1,\infty$
  of the Farey tessellation.}
\label{fig:fargen}
\end{center}
\end{figure}

MAGMA computations also show that we have the explicit isomorphism
$$ N = \langle X,Y,Z \rangle \cong G_2 = \langle x_0,x_1,x_3 \rangle, $$
given by $X \mapsto x_0, Y \mapsto x_1, Z \mapsto x_3$. The normal
group $N \triangleleft PSL(2,\Z_8)$ is of order 32 and the quotient
$\Ss_0 = (S_8,\Ss(S_8))/N$ is an orbifold consisting of two hyperbolic
$(\pi/4,\pi/4,\pi/4)$-triangles (one of them black and the other
white). We conclude from the explicit isomorphism $N \cong G_2$
that the covering procedure discussed in Section \ref{subsec:surfs}
leads to isometric surfaces $\Ss(\Pi_8)$ and $\Ss(T_2)$, and that
$\Pi_8 \subset \Ss(\Pi_8)$ is dual to the tessellation
$(T_2,\Ss(T_2))$.

On the spectral side, the adjacency operators on the graphs $X_2$
and $\Pi_8$ compare as follows:

\medskip

\begin{table}[htbp]
\begin{tabular}{l|c|c|c|c|c|c|c|c|l}
  eigenvalue & $8$ & $6$ & $\sqrt{8}$ & $2$ & $0$ & $-2$ & $-\sqrt{8}$ & $-4$ & Total\\ \hline
  multiplicity in $X_2$ & $0$ & $1$ & $6$ & $6$ & $4$ & $9$ & $6$ & $0$ & $32$ \rule[1.5mm]{0cm}{3mm} \\
  multiplicity in $\Pi_8$ & $1$ & $0$ & $6$ & $0$ & $9$ & $0$ & $6$ & $2$ & $24$ 
  \rule[1.5mm]{0cm}{2mm}
\end{tabular}
\end{table}

We also like to mention that, for $k=2$, the
$(\Delta-Y)$-transformation $X_2 \to T_2$ has a group theoretical
interpretation. There exists a group extension $\widetilde{G_2}$
of $G_2$ by $\Z_2$, generated by involutions $A,B,C$ satisfying $X
= AB$, $Y = BC$ and $Z = CA$, and $T_2$ is the Cayley graph of
$\widetilde{G_2}$ with respect to the generators $A,B,C$. This
group theoretic interpretation of the $(\Delta-Y)$-transformation {\em
  fails} for $k \ge 5$.  In fact, the group
$$ T = \langle A,B,C \mid A^2, B^2, C^2,r_1(AB,BC),r_2(AB,BC),r_3(AB,BC) 
\rangle $$ 
with $r_1,r_2,r_3$ given in \eqref{eq:r1r2r3} is finite and of order
6144. If the introduction of the above involutions $A,B,C$ would lead
to a group extension $\widetilde{G_k}$, then $\widetilde{G_k}$
would have to be of order $2 |G_k|$ and a quotient of $T$ and,
therefore, of order $\le 6144$. However, we have $2 |G_5| = 16384$
in contradiction to the second condition. Thus we do not obtain a
Cayley graph representation of the graphs $T_k$ for $k \ge 5$ via this
procedure.

Let us finally explain why we can no longer have an isomorphism $T_k^* \cong
\Pi_{2^\rho}$ for $k \ge 3$, with $\rho$ appropriately
chosen. Let us assume that $\Pi_{2^\rho} = T_k^*$. Comparison of the vertex
degrees of $\Pi_{2^\rho}$ and $T_k^*$ leads to $\rho = r+1$, with $r$
given in \eqref{eq:rK}. Moreover, we conclude from \eqref{eq:VPiN}
below that $|V(\Pi_{2^\rho}| = 3 \cdot 2^{2\rho -3}$. We know from
Theorem \ref{thm:main} that 
$$ |V(T_k^*)| = |F(T_k)| \ge 3 \cdot 2^{K-r-1}. $$
The condition $|V(\Pi_{2^\rho})| = |V(T_k^*)|$ together with $\rho = r+1$
implies that $3r \ge K$ with $K$ in \eqref{eq:rK}, i.e., 
$$ 3 \lfloor \log_2 k \rfloor + 3 \ge 8 \lfloor k/3 \rfloor + 3\cdot(k\, {\rm mod}\, 3). $$ 
But one easily checks that this inequality holds only for $k=1,2$. (In
the case $k=1$, we have $\Pi_4 = T_1^*$, since $T_1$ is
combinatorially the cube and $\Pi_4$ is the octagon.) This shows that
the graph family $\Pi_N$ cannot contain any of the dual graphs $T_k^*$,
for indices $k \ge 3$.

\subsection{Proof of Corollary \ref{cor:lowerlambda1}}
\label{subsec:corlambda1}

The identity $2g-2 = |V(T_k)|$ between the genus of the surface
$\widehat \Ss(T_k)$ and the number of vertices of the trivalent graph
$T_k$ is easily checked. Moreover, every automorphism of the graph
$T_k$ induces an isometry on $\widehat \Ss(T_k)$. Since the graphs
$T_k$ form a power of coverings with powers of $2$ as covering
indices, the same holds true for the associated surfaces $\widehat
\Ss(T_k)$. Now, \cite{Bu} shows that $\lambda_1(\widehat \Ss(T_k))$
can be estimated from below by a fixed multiple of the isoperimetric Cheeger
constant of $T_k$. The expander property implies that the Cheeger
constants of $T_k$ have a uniform lower positive bound. This finishes
the proof of the Corollary \ref{cor:lowerlambda1}.

\section{The Platonic graphs}

\subsection{Algebraic description of vertices and axes}

Let us briefly recall some algebraic facts from \cite{ISi}. Both
groups $\Gamma=PSL(2,\Z)$ and $PSL(2,\Z_N)$ act on $V(\Pi_N)$ via
$$ \begin{pmatrix} a & b \\ c & d \end{pmatrix} [\lambda,\mu] = 
[a \lambda+b \mu,c \lambda+ d \mu], $$ and there is a 1-1
correspondence between the vertex set $V(\Pi_N)$ and the cosets
$\Gamma / \Gamma_1(N)$. Here, $\Gamma_1(N)$ is the congruence subgroup
given by
$$ 
\Gamma_1(N) = \left\{ \begin{pmatrix} a & b \\ c & d \end{pmatrix} \in \Gamma 
\mid \begin{pmatrix} a & b \\ c & d \end{pmatrix} \equiv \pm \begin{pmatrix}
1 & * \\ 0 & 1 \end{pmatrix} \ {\rm mod}\ N \right\}.
$$
In \cite{ISi}, the set of vertices was partitioned into {\em axes}.
Two vertices belong to the same axis if they have the same stabilizer
in $PSL(2,\Z_N)$. Since $PSL(2,\Z_N)$ acts transitively on $V(\Pi_N)$,
all axes have the same number of vertices. An interesting observation
is that if an element of $PSL(2,\Z_N)$ leaves a vertex $[\lambda,\mu]$
invariant, then any vertex $[\nu,\omega]$ with
$\lambda\omega-\mu\nu=0$ is also invariant under the same
element. Thus the axis containing $[1,0]$ is given by
\begin{equation} \label{eq:princax}
\A_{\rm princ} = \{ [\lambda,0] \mid \gcd(\lambda,N) = 1 \}, 
\end{equation}
and we call this axis the {\em principal axis} of $\Pi_N$. The set of all axes
of $\Pi_N$ is denoted by $A(\Pi_N)$. There is a 1-1 correspondence between
the axes of $\Pi_N$ and the cosets $\Gamma/\Gamma_0(N)$, where
$\Gamma_0(N)$ is the congruence subgroup
$$ 
\Gamma_0(N) = \left\{ \begin{pmatrix} a & b \\ c & d \end{pmatrix} \in
  \Gamma \mid \begin{pmatrix} a & b \\ c & d \end{pmatrix}
  \equiv \begin{pmatrix} * & * \\ 0 & * \end{pmatrix} \ {\rm mod}\ N
\right\}.
$$
From the 1-1 correspondences with the cosets of $\Gamma(N),
\Gamma_1(N), \Gamma_0(N)$, we can immediately obtain the numbers of
directed edges, vertices and axes of $\Pi_N$ as the indices of these
subgroups in $\Gamma$:
\begin{eqnarray}
  |A(\Pi_N)| &=& |\Gamma : \Gamma_0(N)| = N \prod_{p \vert N} 
  \left( 1 + \frac{1}{p} \right), \nonumber \\
  |V(\Pi_N)| &=& |\Gamma : \Gamma_1(N)| = \frac{N^2}{2} \prod_{p \vert N}
  \left( 1 - \frac{1}{p^2} \right), \label{eq:VPiN} \\
  |E(\Pi_N)| &=& \frac{1}{2}|\Gamma : \Gamma(N)| = \frac{N^3}{4} \prod_{p \vert N}
  \left( 1 - \frac{1}{p^2} \right), \nonumber
\end{eqnarray}
where the products run over the distinct prime divisors of $N$, see
for example \cite{Mi}. In particular, for a prime $p$, $\Pi_p$ has
$p+1$ axes, $(p^2-1)/2$ vertices and $p(p^2-1)/4$ undirected edges.

\subsection{Proof of \eqref{eq:asympPiN} in Proposition
  \ref{prop:asymptnonflat}} 
\label{subsec:prop12b}

From 
$$ \chi(\Ss(\Pi_N)) = 2-2g(\Ss(\Pi_N)) = V - E/3 $$
we conclude
$$ g(\Ss(\Pi_N)) = 1 + \frac{N^2(N-6)}{24} \prod_{p \vert N}
\left( 1 - \frac{1}{p^2} \right),$$ 
which immediately implies that
$$ \lim_{N \to \infty} \frac{6g(\Ss(\Pi_N))}{|E(\Pi_N)|} = 1. $$
Note also, that the genus of the surface $\Ss(\Pi_p)$ for a prime $p$
is given by $(p+2)(p-3)(p-5)/24$.

\subsection{Vertex connectivity of $\Pi_p$ and $\Pi_p'$}
\label{subsec:vertconnect}

Let $p$ be a fixed odd prime and $n= (p-1)/2$. The {\em wheel
  structure} of $\Pi_p$ was already discussed in \cite[Thm
2.1]{LR}. Let us present this and other geometric facts in our
terminology. The principal axis of $\Pi_p$ is given by
$$ \A_{\rm princ} = \{ [i,0] \in V(\Pi_p) \mid 1 \le i \le n \}. $$
The vertices of $\A_{\rm princ}$ and their $1$-ring neighbours form a
partition of $V(\Pi_p)$ into $n$ components with $p+1$ vertices
each. We call these components the {\em wheels} of $\Pi_p$, see Figure
\ref{fig:wheels}. The wheel with center $[i,0]$ ($1 \le i \le n$) is
denoted by $W_i$ and is a subgraph of $\Pi_p$ with $p+1$ vertices and
$2p$ edges. We also use the notation $\partial W_i$ for the induced
subgraph with vertex set $V(\partial W_i) = V(W_i) - \{ [i,0] \}$. We
call $\partial W_i$ the boundary of the $i$-th wheel. Note that
$\partial W_i$ is isomorphic to the cyclic graph of $p$ vertices.

Every vertex that is not in $\A_{\rm princ}$ is adjacent to exactly
two vertices of the boundary of any given wheel $W_i$, $1 \le i \le
n$. Indeed, because $PSL(2,\Z_p)$ acts transitively on $V(\Pi_p)$, we
may consider, w.l.o.g., the vertex $[0,1] \in \partial W_1$. The $p-1$
vertices adjacent to $[0,1]$ that are not in $\A_{\rm princ}$ are
$[1,x]$ with $x \in \{1,2,\dots,p-1\}$.
To find the vertices $[1,x]$ in $\partial W_i$, we need to solve
$$ \det \begin{pmatrix} 1 & i \\ x & 0 \end{pmatrix} = \pm 1, $$ 
which has exactly two solution $x = \pm i^{-1}$ (where we think of $i
\in \Z_p$) which correspond to two distinct vertices of $\Pi_p$.

\begin{figure}[htbp]
\begin{center}
\includegraphics[width=0.8\textwidth]{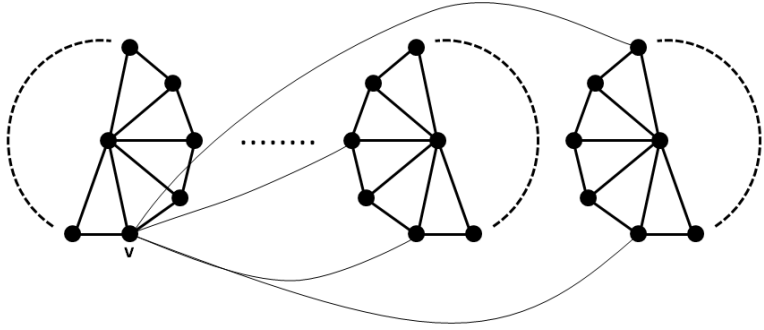}
\caption{$\Pi_p$ consists of $n=(p-1)/2$ wheels. Each vertex at the
  boundary of a wheel is connected with the center of its wheel and
  exactly two points on the boundary of any wheel (including itself).}
\label{fig:wheels}
\end{center}
\end{figure}

\begin{lm} \label{lm:help}
  Let $i,j \in \{1,2,\dots,n\}$. Then we have the following
  facts.
  \begin{itemize}
  \item[(a)] Let $x_1,x_2$ be two different vertices in $\partial W_i$
    and also $y_1,y_2$ be two different vertices in the same set
    $\partial W_i$ ($\{x_1,x_2\} \cap \{y_1,y_2\} \neq \emptyset$ is
    allowed). Then there exists a permutation $\sigma \in \Sym(2)$ and
    two vertex distinct paths $p_1, p_2$ in $\partial W_i$, such that
    $p_1$ connects $x_1$ with $y_{\sigma(1)}$ and $p_2$ connects $x_2$
    with $y_{\sigma(2)}$.
  \item[(b)] Every $x \in \partial W_i$ has precisely two neighbours in
    $\partial W_j$.
  \item[(c)] Assume additionally that $i \neq j$. Then there exists a
    bijective map $\Phi: V(\partial W_i) \to V(\partial W_j)$ such
    that $v \sim \Phi(v)$ for all vertices $v \in \partial W_i$.
  \end{itemize}
\end{lm}

\begin{proof}
  Note that $\partial W_i$ is isomorphic to the cyclic graph of $p$
  vertices. (a) is then a straighforward inspection of all possible
  cases. (b) is already proved by our previous arguments. It
  remains to prove (c): Think of $i,j \in \Z_p-\{0\}$. Then the
  vertices in $\partial W_i$ are of the form $[\mu,i^{-1}]$ and the
  vertices in $\partial W_j$ of the form $[\nu,j^{-1}]$ with $\mu,\nu
  \in \Z_p$. The map $\phi: \Z_p \to \Z_p$, defined by $\phi(\mu) =
  i+ij^{-1}\mu$, is obviously a bijection, and we have $[\mu,i^{-1}]
  \sim [\phi(\mu),j^{-1}]$, finishing the proof.
\end{proof}

Note that the wheel structure is not confined to the choice of the
principal axis. Since the group $PSL(2,\Z_p)$ maps axes to axes and acts
transitively on them, we can choose any axis $\A$ as the centers of
the $n$ wheels, and Lemma \ref{lm:help} is still valid in this
setting.

\medskip

Now we prove that $\Pi_p$ is $p$-vertex-connected. Notice that the
arguments in this proof also give $\diam(\Pi_p) \le 3$ as a by-product.

\begin{proof}
  We will show that for any two vertices of $\Pi_p$, we can find $p$
  vertex disjoint paths connecting them. Then the result will follow from
  Menger's Theorem. 

  Since $PSL(2,\Z_N)$ acts transitively on $V(\Pi_p)$, we can assume
  that the start vertex is $[1,0] \in W_1$. Separating three cases, we
  will find $p$ vertex distinct paths to
  \begin{itemize}
  \item[(i)] the vertices in $\partial W_1$,
  \item[(ii)] the vertices in any $\partial W_j$ with $2 \le j \le n$,
  \item[(iii)] the other vertices in $\A_p$.
  \end{itemize}
  
  Ad (i): Assume that the end vertex is $[\nu,1]$. Then we already
  have three vertex disjoint paths given by
  $$ [1,0] \to [\nu,1], \quad [1,0] \to [\nu\pm1,1] \to [\nu,1]. $$
  We need to find vertex disjoint paths starting with $[1,0] \to
  [\nu\pm i,1]$ and ending at $[\nu,1]$, for $2 \le i \le n$. By
  Lemma \ref{lm:help}(c), we can find two different vertices $x_1,x_2
  \in \partial W_i$ such that $[\nu-i,1] \sim x_1$ and
  $[\nu+i,1] \sim x_2$. By Lemma \ref{lm:help}(b), $[\nu,1]$ has
  two different neighbours $\{y_1,y_2\}$ in $\partial W_i$. We now use
  Lemma \ref{lm:help}(a) to complete the paths.

  Ad (ii): We assume $p \ge 5$, for otherwise there is nothing to
  prove. Let us assume that the end vertex is in $\partial W_i$
  with $2 \le i \le n$, and let us denote this vertex by $w
  \in \partial W_i$. Let $v_-, v_+ \in \partial W_1$ be the two
  neighbours of $w$ in the first wheel. Choose three different
  vertices $v_1, v_2, v_3 \in \partial W_1 - \{v_-,v_+\}$, and use
  Lemma \ref{lm:help}(c) to find three different vertices $w_1,w_2,w_3
  \in \partial W_i - \{w\}$ such that $v_j \sim w_j$ for $1 \le j \le
  3$. W.l.o.g., we can assume that the pair $\{w_1,w_3\}$ separates
  $w_2$ and $w$ within $\partial W_i$. Let $q_1, q_3$ be the two
  vertex disjoint paths in $\partial W_i - \{w_2\}$ connecting $w$
  with $w_1$ and $w_3$, respectively. Then we already have five vertex
  disjoint paths given by
  $$ 
  [1,0] \to v_\pm \to v, \quad [1,0] \to v_2 \to w_2 \to [i,0] \to w,
  $$
  and 
  $$
  [1,0] \to v_1 \to w_1 \stackrel{q_1}{\longrightarrow} w, \quad
  [1,0] \to v_3 \to w_3 \stackrel{q_3}{\longrightarrow} w.
  $$
  Notice that for any wheel $W_j$ with $j \not\in \{1,i\}$, we have
  not yet used any edges with one vertex in $\partial W_j$. We will
  see that every such wheel allows us to create two more vertex
  disjoint paths from $[1,0]$ to $w$, finishing this case. Let
  $y_1,y_2$ be the two different vertices in wheel $\partial W_j$
  adjacent to $w$. Choose two different vertices $v', v'' \in \partial
  W_1$ which have not been used yet and associate to them two
  different vertices $x_1,x_2 \in \partial W_j$ such that $v' \sim
  x_1$ and $v'' \sim x_2$, using Lemma \ref{lm:help}(c). Then we can
  use Lemma \ref{lm:help}(a) to complete the paths within $\partial
  W_j$.

  Ad (iii): This is the easiest case. Assume that the end vertex is
  $[i,0] \in W_i$ with $2 \le i \le n$. We use the bijection
  $\Phi: V(\partial W_1) \to V(\partial W_i)$ in Lemma \ref{lm:help}(c)
  to create the $p$ vertex disjoint paths 
  $$ [1,0] \to [0,\mu] \to \Phi([0,\mu]) \to [i,0] $$
  with $1 \le \mu \le p$.
\end{proof}

Next, we present the proof that $\Pi_p'$ is
$(p-1)$-vertex-connected. In contrast to the previous proof, the
arguments given here do not imply that $\diam(\Pi_p') \le 3$.

\begin{proof}
  Let $v,w \in \Pi_p'$ be two different vertices with $v \in \partial
  W_i$ and $w \in \partial W_j$. We consider the two cases $i = j$ and
  $i \neq j$ separately:

  Case $i=j$: Obviously, we can choose two vertex disjoint paths
  within $\partial W_i$ to connect $v$ and $w$. Next, we show that
  every wheel $\partial W_j$ with $j \neq i$ gives rise to two
  additional vertex disjoint paths. Let $x_1,x_2 \in \partial W_j$ be
  the two distinct neighbours of $v$, and $y_1,y_2 \in \partial W_j$
  be the two distinct neighbours of $w$. Then we can use Lemma
  \ref{lm:help}(a) to complete the paths within $\partial W_j$.
  
  Case $i \neq j$: Let $w_1,w_2 \in \partial W_j$ be the neighbours of
  $v$ and $v_1,v_2 \in \partial W_i$ be the neighbours of $w$. Then,
  using only additional edges in $\partial W_i \cup \partial W_j$, we
  can find four vertex disjoint paths $v \to \cdots \to v_k \to w$, $v
  \to w_k \to \cdots \to w$ (for $k=1,2$). Again, every wheel $W_l$
  with $l \not\in \{i,j\}$ will give rise to two more vertex disjoint
  paths. Let $x_1,x_2 \in \partial W_l$ be the neighbours of $v$, and
  $y_1, y_2 \in \partial W_l$ be the neighbours of $w$. Use Lemma
  \ref{lm:help}(c) to complete the paths within $\partial W_l$.
\end{proof}

Finally, we prove $\diam(\Pi_p') = 3$.

\begin{proof}
  Let us first confirm that any two different vertices in the same
  wheel can be connected by a path of length $2$: Let $1 \le i \le n$
  and $[\mu,i^{-1}],[\nu,i^{-1}] \in \partial W_i$ (thinking of $i \in
  \Z_p$) be the two vertices. The required path is then given by
  $$ [\mu,i^{-1}] \to 
  [ 2(\mu-\nu)^{-1}\mu i -i,2(\mu-\nu)^{-1}] \to [\nu,i^{-1}]. $$ 
  Now choose two vertices $v \in \partial W_i$ and $w \in \partial
  W_j$ on different wheels. Let $v' \in \partial W_i$ be one of the
  two neighbours of $w$ in the $i$-th wheel. Connecting $v$ and $v'$
  by a path of length 2 (as shown before) implies that $d(v,w) \le
  d(v,v') + 1 \le 3$.
\end{proof}

Note that the graph $\Pi_p'$ was initially defined as the induced
subgraph of $\Pi_p$ with vertex sets $V(\Pi_p) - \A_{\rm princ}$. Alternatively,
$\Pi_p'$ can also described as the Cayley graph $\Cay(U_p,S)$ with
$$ U_p = \left\{ \begin{pmatrix} * & * \\ 0 & * \end{pmatrix} 
  \in PSL(2,\Z_p) \right\} \cong \Gamma_0(p) / \Gamma(p) \ \text{and}
\, S = \left\{ \begin{pmatrix} * & 1 \\ 0 & * \end{pmatrix} \in U_p
\right\}. $$ 
The vertices $[\lambda,\mu] \in V(\Pi_p')$ (with non-vanishing second
coordinate $\mu$) are then identified with the matrices
$\begin{pmatrix} \mu^{-1} & \lambda \\ 0 & \mu \end{pmatrix} \in U_p$.

\subsection{Ramanujan properties ''without number theory''}
\label{subsec:ramanujan}

As before, we assume that $p$ is a fixed odd prime and $n=
(p-1)/2$. The considerations of the previous section show also that
$\Pi_p$ is a $n$-fold covering $\pi: \Pi_p \to K_{p+1}$ of the
complete graph $K_{p+1}$, where the preimages $\pi^{-1}(v)$ correspond
to the axes of $\Pi_p$. It is useful to think of the vertices in
$K_{p+1}$ as the points in the finite projective line over the field
$\Z_p$, i.e., $V(K_{p+1}) = \{0,1,\dots,p-1,\infty\}$ and the covering
map is then given, algebraically, by
$$ \pi([\lambda,\mu]) = \lambda\mu^{-1}, $$
with the usual convention $\infty^{-1} = 0$ and $0^{-1} =
\infty$. In particular, we have $\A_{\rm princ} = \pi^{-1}(\infty)$.
Note that $PSL(2,\Z_N)$ acts also on the vertices of $K_{p+1}$ via 
$$ \begin{pmatrix} \alpha & \beta \\ \gamma & \delta \end{pmatrix} z = 
(\alpha z + \beta)(\gamma z + \delta)^{-1}. $$ 
One easily checks that $\pi( g v ) = g \pi(v)$ for all
$g \in PSL(2,\Z_N)$ and $v \in V(\Pi_p)$.

Let us now explicitely derive the spectra of the graphs $\Pi_p$ and
$\Pi_p'$. We will use the following notation: For a linear operator
$T$ on a finite dimensional vector space, we denote the eigenspace of
$T$ to the eigenvalue $\lambda$ by $\E(T,\lambda)$.

We start with a ''number theory free'' proof of Theorem 4.2 in 
\cite{Gunnells}, using the covering $\pi: \Pi_p \to K_{p+1}$.

\begin{proof}
  Every eigenfunction $f$ of $K_{p+1}$ gives rise to an eigenfunction
  $F: V(\Pi_p) \to \C$ of the same eigenvalue via $F(v) =
  f(\pi(v))$. The spectrum of the adjacency operator on $K_{p+1}$ is
  given by (see, e.g., \cite[p. 17]{Biggs})
  $$ \sigma(K_{p+1}) = \{ p, \underbrace{-1,\dots,-1}_{\text{p times}} \}. $$
  This implies that $\sigma(\Pi_p)$ contains the eigenvalue $p$ with
  multiplicity one and the eigenvalue $-1$ with multiplicity $\ge p$.

  Our next aim is to prove that the eigenspace $\E(A^2,p)$ of the
  square of the adjacency operator on $\Pi_p$ has dimension
  $(p+1)(p-3)/2$. Let $f: V(\Pi_p) \to \C$ be a function satisfying
  \begin{equation} \label{eq:eigp}
  A^2 f(v) = p f(v) \qquad \text{for all $v \in V(\Pi_p)$.} 
  \end{equation}
  Note that \eqref{eq:eigp} can be viewed as a homogenous system
  of $(p^2-1)/2$ linear equations. The key observation is that all
  linear equations corresponding to vertices of the same axis
  coincide, i.e., we end up with only $p+1$ linear independent
  homogeneous equations (since $p+1$ equals the number of axes),
  showing that the eigenspace has dimension at least 
  $$ |V(\Pi_p)| - (p+1) = \frac{p^2-1}{2} - (p+1) = \frac{(p+1)(p-3)}{2}. $$ 
  Indeed, since $PSL(2,\Z_p)$ acts transitively on the vertices, we
  only need to show that the linear equations of \eqref{eq:eigp}
  corresponding to the vertices in the principal axis $\A_p$
  coincide. Recall that $S_p$ has the wheel-structure given in Figure
  \ref{fig:wheels}. Let $v \in \A_p$. Then we have
  \begin{equation} \label{eq:a2cond}  
  A^2f(v) = p f(v) + 2 \sum_{i=1}^n \sum_{w \in \partial W_i} f(w),
  \end{equation}
  since there are exactly $p$ paths of length 2 from $v$ to itself, no
  paths of length 2 from the centers of all the other wheels to $v$,
  and for every $w \in \cup_i \partial W_i$ there are exactly $2$
  paths from $w$ to $v$ of length 2, because of Lemma
  \ref{lm:help}(b). Note that the combination of \eqref{eq:eigp} and
  \eqref{eq:a2cond} simplifies to
  $$ \sum_{i=1}^n \sum_{w \in \partial W_i} f(w) = 0, $$
  independently of the choice of $v \in \A_p$. This shows that $\dim
  \E(A^2,p) \ge (p+1)(p-3)/2$. Adding up the multiplicities of all
  eigenvalues, we see that $\dim \E(A^2,p) = (p+1)(p-3)/2$.
  
  If $f_1,\dots,f_K$ span the space $\E(A^2,p)$, then the $2K$ functions
  $$ \sqrt{p}f_1 \pm A f_1,\dots,\sqrt{p}f_K \pm A f_K $$
  are eigenfunctions of $A$ to the eigenvalues $\pm \sqrt{p}$, and
  they also span $\E(A^2,p)$. This shows that we have
  $$ \E(A^2,p) = \E(A,\sqrt{p}) \oplus \E(A,-\sqrt{p}). $$
  Finally, the equality
  $$ \dim \E(A,\sqrt{p}) = \dim \E(A,-\sqrt{p}) = \frac{(p+1)(p-3)}{4} $$
  follows from Lemma \ref{lm:eqdim} below.
\end{proof}

\begin{lm} \label{lm:eqdim}
  Let $T$ be a square matrix with rational entries and $K$ be a positive
  integer which is not a square. Then we have
  $$ \dim \E(T,\sqrt{K}) =\dim \E(T,-\sqrt{K}). $$ 
\end{lm}

\begin{proof}
  The proof is based on the fact that $\sqrt{K}$ is irrational. Let
  $p(z) \in \Q[z]$ be the characteristic polynomial of $T$. We split $p(z)$
  into its even and odd part, i.e.,
  $$ p(z) = p_{even}(z) + p_{odd}(z) z, $$
  with even polynomials $p_{even}(z),p_{odd}(z)$. Note that we have
  $$ p(\sqrt{K}) = p_{even}(\sqrt{K}) + p_{odd}(\sqrt{K}) \sqrt{K} $$
  and $p_{even}(\sqrt{K}), p_{odd}(\sqrt{K}) \in \Q$. Therefore, if
  $\sqrt{K}$ is a root of $p(z)$, then $\sqrt{K}$ is also a root of
  both polynomials $p_{even}(z)$ and $p_{odd}(z)$, separately. This
  implies that $-\sqrt{K}$ is also a root of $p(z)$. We can then split
  off the factor $z^2 - K$ from $p(z)$, and repeat the procedure with
  the remaining polynomial.
\end{proof}

Next we derive the spectrum of the modified graph $\Pi_p'$, using the
$n$-fold covering map $\pi: \Pi_k' \to K_p$ and the wheel-structure,
which partitions the vertex set $V(\Pi_p')$ into $n$ disjoint sets
$\partial W_i$ of $p$ vertices, each. This will finish the proof of
Theorem \ref{thm:platonic}.

\begin{proof} The proof of the spectral statements in Theorem
  \ref{thm:platonic} proceeds in steps.

  (i) Let ${\mathcal W}$ be the vector space of all functions which
  are constant on the wheels. We first introduce a basis of
  eigenfunctions of this vector space. Let $\zeta_n = e^{2\pi i/n}$
  and, for $0 \le j \le n-1$, define
  $$ f_j(v) = \zeta_n^{ij} \quad \text{if $v \in \partial W_i$.} $$
  Note that $f_0$ is the constant function to the eigenvalue $p-1$.
  It is easily checked that $A f_j = 0$ for $j \ge 1$. Since these
  functions are linearly independent, they form a basis of ${\mathcal
    W}$. Moreover, we have $\dim \E(A,0) \ge n-1=(p-3)/2$.

  (ii) Let ${\mathcal V}$ be the vector space of all functions which
  are constant along all axes. Every such function is a lift $F(v) =
  f(\pi(v))$ of a function $f$ on $K_p$. Note that eigenfunctions of
  $K_p$ are lifted to eigenfunctions to the same eigenvalue, so
  ${\mathcal V}$ can be viewed as the span of a constant function
  and $p-1$ linear independent eigenfunctions to the eigenvalue
  $-1$. In particular, we have $\dim \E(A,-1) \ge p-1$.
  
  (iii) Note that ${\mathcal W} \cap {\mathcal V} = {\rm span}(f_0)$.
  By the orthogonality of eigenfunctions, it only remains to study the
  eigenfunctions in the orthogonal complement $({\mathcal W} +
  {\mathcal V})^\bot$ of dimension
  $$ |V(\Pi_p')| - (\dim {\mathcal W} + \dim {\mathcal V}) + 1 = 
  \frac{(p-1)(p-3)}{2} = K. $$ 

  Let $g_1,\dots,g_K$ be a basis of this orthogonal complement by
  eigenfunctions with $A g_i = \lambda_i g_i$. We now extend each
  $g_i$ trivially to a function $\widetilde g_i$ on $S_p$ by setting
  $\widetilde g_i(v) = 0$ for all $v \in \A_p$. Note that these
  extensions are eigenfunctions of the Platonic graph $\Pi_p$ to the
  same eigenvalue, i.e., $A \widetilde g_i = \lambda_i \widetilde
  g_i$. Therefore, we must have $\lambda_i \in \{p+1,-1,\pm
  \sqrt{p}\}$. As discussed in the previous proof, the span of the
  eigenfunctions of $\Pi_p$ to the eigenvalues $-1$ and $p+1$ is
  obtained via lifting the eigenfunctions of $K_{p+1}$, and the
  restriction of these functions to $\Pi_p'$ must therefore lie in
  $\mathcal V$. This shows that we must have $\lambda_i = \pm
  \sqrt{p}$.

  (iv) Adding up the multiplicities of all eigenvalues, we
  conclude that
  \begin{eqnarray*}
  \dim \E(A,\sqrt{p}) \oplus \E(A,-\sqrt{p}) &=& (p-1)(p-3)/2, \\
  \dim \E(A,0) &=& (p-3)/2, \\
  \dim \E(A,-1) &=& p-1. 
  \end{eqnarray*}
  We finally obtain 
  $$ \dim \E(A,\sqrt{p}) = \dim \E(A,-\sqrt{p}) = \frac{(p-1)(p-3)}{4}, $$
  by applying, again, Lemma \ref{lm:eqdim}.
\end{proof}

\end{document}